\documentclass[11pt,reqno,twoside]{amsart}

\usepackage{amsfonts}
\usepackage{amsmath,amssymb}
\usepackage[dvips,bottom=1.4in,right=1in,top=1in, left=1in]{geometry}
\usepackage{epic}
\usepackage{pstricks}
\usepackage{graphicx}

\newtheorem{theorem}{Theorem}[section]

\newtheorem{definition}[theorem]{Definition}

\newtheorem{proposition}[theorem]{Proposition}

\newtheorem{remark}[theorem]{Remark}
\numberwithin{equation}{section}

\def\RR{{\mathbb{R}}}
\def\NN{{\mathbb{N}}}

\def\CC{{\mathbb{C}}}

\def\Om{\Omega}

\def\bOm{\overline{\Om}}
\def\pOm{\partial\Omega}

\newcommand{\norm}[2]{{\left\|#1\right\|}_{#2}}

\newcommand{\qint}{\int_{\Omega_T}}
\newcommand{\sint}{\int_{\Sigma_T}}

\newcommand{\dt}[1]{\mathbf{DT}_{#1}}
\newcommand{\bt}[1]{\mathbf{BT}_{#1}}

\newcommand{\ZQ}{Z_{\Omega_T}}
\newcommand{\ZS}{Z_{\Sigma_T}}
\newcommand{\TT}{\mathcal{T}}
\newcommand{\sop}{\mathcal{S}}
\newcommand{\LB}{\Delta_{\Gamma}}
\newcommand{\GB}{\nabla_{\Gamma}}

\title[Boundary controllability]{Boundary null controllability for a heat equation with general dynamical  boundary conditions}

\author{Umberto Biccari}
\address{U.~Biccari, Basque Center for Applied Mathematics (BCAM), Alameda Mazarredo 14. 48009 Bilbao Basque Country  (Spain)}
\email{ubiccari@bcamath.org}

\author{Mahamadi Warma}
\address{M.~Warma, University of Puerto Rico, Faculty of Natural Sciences,
Department of Mathematics (Rio Piedras Campus), PO Box 70377 San Juan PR
00936-8377 (USA)}
\email{mahamadi.warma1@upr.edu, mjwarma@gmail.com}

\thanks{The work of the authors is partially supported by the Air Force Office of Scientific Research under the Award No: FA9550-15-1-0027}



\begin{document}

\begin{abstract}
Let $\Omega\subset\RR^N$ be a bounded open set with Lipschitz continuous boundary $\Gamma$. Let $\gamma>0$, $\delta\ge 0$ be real numbers and $\beta$ a nonnegative measurable function in $L^\infty(\Gamma)$.
Using some suitable Carleman estimates, we show that the linear heat equation $\partial_tu - \gamma\Delta u  = 0$ in $\Omega\times(0,T)$ with the non-homogeneous general dynamic boundary conditions $\partial_tu_{\Gamma} -\delta\Delta_\Gamma u_{\Gamma}+ \gamma\partial_{\nu}u + \beta u_{\Gamma} = g$  on $\Gamma\times(0,T)$ is always null controllable from the boundary for every $T>0$ and initial data $(u_0,u_{\Gamma,0})\in L^2(\Omega)\times L^2(\Gamma)$.
\end{abstract}

\maketitle

\section{Introduction and main results}

Let $\Omega\subset\mathbb R^N$ be a bounded open set with Lipschitz continuous boundary $\Gamma:=\partial\Omega$.
In the present paper we consider the following heat equation with dynamical boundary conditions
\begin{align}\label{heat_dbc}
	\begin{cases}
\partial_tu - \gamma\Delta u  = 0 &\mbox{ in }\; \Omega\times(0,T):=\Omega_T\\
\partial_tu_{\Gamma} -\delta\Delta_\Gamma u_{\Gamma}+ \gamma\partial_{\nu}u + \beta(x)u_{\Gamma} = g \;\;&\mbox{ on }\;\Gamma\times(0,T):=\Sigma_T\\
\left.(u,u_{\Gamma})\right|_{t=0} = (u_0,u_{\Gamma,0}) \;\;&\mbox{ in }\; \Omega\times\Gamma.
\end{cases}
\end{align}
Here, $\gamma>0$ and  $\delta\ge 0$ are real numbers, $\beta$ is a nonnegative measurable function which belongs to $L^{\infty}(\Gamma)$, $\partial_\nu u$ is the normal derivative of $u$ and $u_{\Gamma}$ denotes the trace on $\Gamma$ of the function $u$, whereas $u_0\in L^2(\Omega)$, $u_{\Gamma,0}\in L^2(\Gamma)$ and $\Delta_\Gamma$ denotes the Laplace-Beltrami operator on $\Gamma$. We emphasize that $u_{\Gamma,0}$ is not necessarily the trace of $u_0$, since we do not assume that $u_0$ has a trace. But if $u_0$ has a well-defined trace on $\Gamma$, then the trace must coincide with $u_{\Gamma,0}$. For example if the one dimensional case $N=1$, since solutions of \eqref{heat_dbc} are continuous on $\bOm$, then in that case we have that $u_{\Gamma,0}$ is the trace of $u_0$.

Several authors have studied the existence, uniqueness and the regularity of solutions to the system \eqref{heat_dbc}. We refer for example to the papers \cite{AMPR,FGGR1,FGGR,GW,KN,MS2013,VV,War} and their references. This type of boundary conditions has been also called generalized Wentzell or generalized Wentzell-Robin boundary conditions.

The main concern in the present paper is the investigate the null controllability of the system \eqref{heat_dbc} from the boundary, that is, given $T>0$ and initial data $(u_0,u_{\Gamma,0})\in L^2(\Omega)\times L^2(\pOm)$, is there a control function $g\in L^2(\Sigma_T)$ such that the unique mild solution (see Definition \ref{def-sol} below) satisfies $u(\cdot,T)=0$ in $\Omega$ and $u_{\Gamma}(\cdot,T)=0$ on $\Gamma$? 

The interior null controllability of the system \eqref{heat_dbc} for the case $\delta>0$ has been recently studied in \cite{MS2013}, that is, the case where $g=0$ and the first equation is replaced by $\partial_tu - \gamma\Delta u  = f|_{\omega}$  in $\Omega_T$. The authors have shown that if $\delta>0$, then for every open set $\omega\Subset\Omega$, $T>0$, and initial data $(u_0,u_{\Gamma,0})\in L^2(\Omega)\times L^2(\pOm)$, there is a control function $f\in L^2(\Omega_T)$ such that the unique mild solution of the associated system satisfies $u(\cdot,T)=0$ in $\Omega$ and $u_{\Gamma}(\cdot,T)=0$ on $\Gamma$. On the other hand, in the one dimensional case, that is, $\Omega=(0,1)$, it has been proved in \cite{KN} that the system  \eqref{heat_dbc} is approximately controllable, that is, for every initial data $(u_0,u_{0,0},u_{1,0})\in L^2(0,1)\times \CC^2$, $T>0$ and $\varepsilon>0$, there exists a control function $g\in L^2(0,T)$ such that the unique mild solution $u$ satisfies
\begin{align*}
\|u(\cdot,T)-u_0\|_{L^2(0,1)}+|u(0,T)-u_{0,0}|+|u(1,T)-u_{1,0}|<\varepsilon.
\end{align*}
We emphasize that the same approximate controllability of the system \eqref{heat_dbc} can be proved in the $N$-dimensional setting, but this will be a simple consequence of the stronger result obtained in the present paper.
Since interior null controllability of a system does not imply the null controllability of the system from the boundary, and as approximate controllability does not imply null controllability (but the converse is always true), we have that the results obtained in the present papers will complete the ones contained in \cite{MS2013} and will trivially imply the ones obtained in 
\cite{KN}. Our main result (see Theorem \ref{dbc_control_thm} below) states that the system is null controllable from the boundary for every $T>0$ and initial data $(u_0,u_{\Gamma,0})\in L^2(\Omega)\times L^2(\pOm)$. Its proof  is based on suitable Carleman estimates (Theorem \ref{carleman_thm}) for solutions of the adjoint system associated with \eqref{heat_dbc} which are also used to establish an observability inequality for solutions of the adjoint system. The obtained observability ineuqality is as usual equivalent to the null controllability of the system. We also notice that our results also include the case $\delta=0$, that is, when there is no  surface diffusion at the boundary. This case has not been considered in the study of the interior controllability in \cite{MS2013}.

\subsection{The functional setup}
Let $\Omega\subset\RR^N$ be a bounded open set with Lipschitz continuous boundary $\Gamma$.
For $r,q\in \lbrack 1,\infty ]$ with $1\leq r,q<\infty $ or $r=q=\infty $ we
endow the Banach space 
\begin{equation*}
\mathbb{X}^{r,q}(\overline{\Omega }):=L^{r}(\Omega )\times L^{q}(\Gamma)=\{(f,g):\;f\in L^{r}(\Omega ),\;g\in L^{q}(\Gamma)\}
\end{equation*}%
with the norm 
\begin{equation*}
\Vert (f,g)\Vert _{\mathbb{X}^{r,q}(\overline{\Omega })}:=
\begin{cases}
\displaystyle\Vert f\Vert
_{L^{r}(\Omega )}+\Vert g\Vert _{L^{q}(\Gamma)}\; &\mbox{ if }\; 1\leq r\ne q<\infty \\
\displaystyle(\Vert
f\Vert _{L^{r}(\Omega )}^{r}+\Vert g\Vert _{L^{r}(\Gamma)}^{r})^{%
\frac{1}{r}}\; &\mbox{ if }\; 1\leq r=q<\infty 
\end{cases}
\end{equation*}%
and 
\begin{equation*}
\Vert (f,g)\Vert _{\mathbb{X}^{\infty ,\infty }(\overline{\Omega })}:=\max
\{\Vert f\Vert _{L^{\infty }(\Omega )},\Vert g\Vert _{L^{\infty }(\Gamma)}\}.
\end{equation*}%
We will simple write $\mathbb{X}^{r}(\overline{\Omega }):=\mathbb{X}^{r,r}(\overline{\Omega })$.  We notice that $\mathbb X^r(\bOm)$ can be identified with the Lebesgue space $L^r(\bOm,\mu)$ where the measure $\mu$ on $\bOm$ is defined for every measurable set $B\subset\bOm$ by
\begin{align*}
\mu(B):=|\Omega\cap B|+\sigma(B\cap\Gamma).
\end{align*}
Here $|\cdot|$ denotes the $N$-dimensional Lebesgue measure on $\Omega$ and $\sigma$ is the $(N-1)$-dimensional Lebesgue surface measure on $\Gamma$. In addition we have that $\mathbb X^2(\bOm)$ is a Hilbert space with the scalar product
\begin{align*}
	\langle (f_1,g_1),(f_2,g_2)\rangle_{\mathbb X^2(\bOm)} = \langle f_1,f_2\rangle_{L^2(\Omega)} + \langle g_1,g_2\rangle_{L^2(\Gamma)}
	=\int_{\Omega}f_1f_2\;dx+\int_{\Gamma}g_1g_2\;d\sigma.
\end{align*}
Let
\begin{align*}
W^{1,2}(\Omega)=\left\{u\in L^2(\Omega):\;\int_{\Omega}|\nabla u|^2\;dx<\infty\right\}
\end{align*}
endowed with the norm
\begin{align*}
\|u\|_{W^{1,2}(\Omega)}:=\left(\int_{\Omega}|u|^2\;dx+\int_{\Omega}|\nabla u|^2\;dx\right)^{\frac 12}
\end{align*}
be the first order Sobolev space. The space $W^{1,2}(\Gamma)$ is defined similarly by
\begin{align*}
W^{1,2}(\Gamma)=\left\{u\in L^2(\Gamma):\; \int_{\Gamma}|\nabla_\Gamma u|^2\;d\sigma<\infty\right\},
\end{align*}
where $\nabla_\Gamma$ denotes the Riemannian gradient (see Section \ref{sec-LB} below).
We also introduce the fractional order Sobolev space
\begin{align*}
W^{\frac 12,2}(\Gamma):=\left\{u\in L^2(\Gamma):\;\int_{\Gamma}\int_{\Gamma}\frac{|u(x)-u(y)|^2}{|x-y|^{N}}\;d\sigma_xd\sigma_y<\infty\right\}
\end{align*}
and we endow it with the norm
\begin{align*}
\|u\|_{W^{\frac 12,2}(\Gamma)}=\left(\int_{\Gamma}|u|^2\;d\sigma+\int_{\Gamma}\int_{\Gamma}\frac{|u(x)-u(y)|^2}{|x-y|^{N}}\;d\sigma_xd\sigma_y\right)^{\frac 12}.
\end{align*}
Since $\Omega$ is assumed to have a Lipschitz continuous boundary, then we have the continuous embedding $W^{1,2}(\Omega)\hookrightarrow W^{\frac 12,2}(\Gamma)$. That is, there exists a constant $C>0$ such that 
for every $u\in W^{1,2}(\Omega)$, we have
\begin{align*}
\|u_{\Gamma}\|_{W^{\frac 12,2}(\Gamma)}\le C\|u\|_{W^{1,2}(\Omega)}. 
\end{align*}
For a real number $\delta\ge 0$ we let
\begin{align*}
\mathbb W_\delta^{1,2}(\bOm):=\left\{U:=(u,u_{\Gamma}):\; u\in W^{1,2}(\Omega),\;\delta u\in W^{1,2}(\Gamma)\right\},
\end{align*}
and we endow it with the norm
\begin{align*}
\|(u,u_{\Gamma})\|_{\mathbb W_\delta^{1,2}(\bOm)}:=\left(\|u\|_{W^{1,2}(\Omega)}^2+\|u_{\Gamma}\|_{W^{1,2}(\Gamma)}^2\right)^{\frac 12}\;\;\mbox{ if }\;\delta>0,
\end{align*}
 and
\begin{align*}
\|(u,u_{\Gamma})\|_{\mathbb W_0^{1,2}(\bOm)}:=\left(\|u\|_{W^{1,2}(\Omega)}^2+\|u_{\Gamma}\|_{W^{\frac 12,2}(\Gamma)}^2\right)^{\frac 12} \;\;\mbox{ if }\;\delta=0.
\end{align*}
By definition,  for every $\delta\ge 0$, we have the continuous embedding $\mathbb W_\delta^{1,2}(\bOm)\hookrightarrow\mathbb X^2(\bOm)$.

\subsection{The Laplace-Beltrami operator}\label{sec-LB}
We present some basic notion on the Laplace-Beltrami operator. 
Recall that the boundary $\Gamma$ of the open set $\Omega\subset\RR^N$ can be viewed as a Riemannian manifold endowed with the natural metric inherited from $\RR^N$, given in local coordinates by $\sqrt{\det G}\;dy_1\cdots dy_{N-1}$, where $G=(g_{ij})$ denotes the metric tensor. Let $\nabla_\Gamma$ be the Riemannian gradient. Then the so called Laplace-Beltrami  operator $\LB$ can be at first defined for $u,v\in C^2(\Gamma)$ by the formula
\begin{align}\label{surf_div_thm}
-	\int_{\Gamma} v\LB u\, d\sigma = \int_{\Gamma} \langle\GB u,\GB v\rangle_{\Gamma}\,d\sigma,
\end{align}
where $\langle\cdot,\cdot\rangle_{\Gamma}$ is the Riemannian inner product of tangential vectors on $\Gamma$. Throughout the following we shall just denotes $ \langle\GB u,\GB v\rangle_{\Gamma}=\GB u\cdot\GB v$.
Letting $(g^{ij})=(g_{ij})^{-1}$, then $\LB$ is given in local coordinates by
\begin{align}\label{LB_def}
\LB u= \frac{1}{\sqrt{\det G}}\sum_{i,j=1}^{N-1} \frac{\partial}{\partial y_i}\left(\sqrt{\det G}\, g^{ij}\frac{\partial u}{\partial y_j}\right).
\end{align} 
Using \eqref{LB_def} we have that $\LB$ can be considered as a bounded linear operator from $W^{s+2,2}(\Gamma)$ to $W^{s,2}(\Gamma)$, for any $s\in\RR$. This implies that the formula \eqref{surf_div_thm} extends by density to $u,v\in W^{1,2}(\Gamma)$, where the integral in the left-hand side is to be interpreted in the distributional sense, that is, as $\LB u\in W^{-1,2}(\Gamma):=(W^{1,2}(\Gamma))^\star$. Letting 
\begin{align*}
D(\LB):=\{u\in W^{1,2}(\Gamma),\;\LB u\in L^2(\Gamma)\},
\end{align*}
we have that $-\LB$ is a self-adjoint and nonnegative operator on $L^2(\Gamma)$ (see, e.g., \cite[p. 309]{Tay}). This implies that $\LB$  generates a strongly continuous semigroup on $L^2(\Gamma)$ which is also analytic.  If $\Gamma$ is smooth (say of class $C^2$), then one can show that $D(\LB)=W^{2,2}(\Gamma)$. 

Finally, for the sake of completeness, we mention that in this paper we will never use the local formula \eqref{LB_def} for the Laplace-Beltrami operator, but rather the so-called surface divergence theorem given in \eqref{surf_div_thm}.
Moreover, we recall the following interpolation inequality (see \cite[Theorem 1.3.3]{Trieb}). There exists a constant $C>0$ such that the estimate
\begin{align}\label{GB_interpol}
	\norm{u}{W^{1,2}(\Gamma)}^2\leq C\norm{u}{L^2(\Gamma)}\norm{u}{D(\LB)}
\end{align}
holds for every $u\in D(\LB)$, where the space $D(\LB)$ is endowed with the graph norm defined by 
\begin{align*}
\|u\|_{D(\LB)}=\|u\|_{L^2(\Gamma)}+\|\LB u\|_{L^2(\Gamma)}.
\end{align*}
For more details on this topic we refer to \cite[Chapter 3]{Jost} or \cite[Sections 2.4 and 5.1]{Tay} and their references.

\subsection{The well-posedness}

In this (sub)section we discuss the well-posedness of the system \eqref{heat_dbc}.  
First, let $\mathcal E_\delta$ be the bilinear symmetric form on $\mathbb X^2(\bOm)$ with domain $D(\mathcal E_\delta):=\mathbb W_\delta^{1,2}(\bOm)$ and given for every $U:=(u,u_{\Gamma}), V:=(v,v_{\Gamma})\in\mathbb W_\delta^{1,2}(\bOm)$ by
\begin{align}
\mathcal E_\delta(U,V)=\gamma\int_{\Omega}\nabla u\cdot\nabla v\;dx+\delta\int_{\Gamma}\nabla_\Gamma u_\Gamma\cdot\nabla_\Gamma v_\Gamma\;d\sigma+\int_{\Gamma}\beta(x) u_\Gamma v_\Gamma\;d\sigma.
\end{align}
We assume that $\beta\in L^\infty(\Gamma)$ is measurable and there exists a constant $\beta_0>0$ such that
\begin{align}\label{beta}
\beta(x)\ge \beta_0\;\;\sigma\mbox{-a.e. on }\;\Gamma.
\end{align}
It is well-known that the form $\mathcal E_\delta$ is closed in $\mathbb X^2(\bOm)$, continuous and elliptic. 
Under the assumption \eqref{beta}, it is also coercive, that is, there is a constant $C>0$ such that for every $U=(u,u_\Gamma)\in\mathbb W_\delta^{1,2}(\bOm)$, we have
\begin{align}\label{sobo}
\|U\|_{\mathbb X^2(\bOm)}^2=\int_{\Omega}|u|^2\;dx+\int_{\Gamma}|u_\Gamma|^2\;d\sigma\le C\mathcal E_\delta(U,U).
\end{align}
Second, let $A_\delta$ be the linear self-adjoint operator in $\mathbb X^2(\bOm)$ associated with $\mathcal E_\delta$ in the sense that
\begin{equation}\label{op-A}
\begin{cases}
D(A_\delta)=\{U\in \mathbb W_\delta^{1,2}(\bOm),\;\exists\;F\in \mathbb X^2(\bOm),\;\;\mathcal E_\delta(U,\Phi)=\langle F,\Phi\rangle_{\mathbb X^2(\bOm)},\;\forall\;\Phi\in \mathbb W_\delta^{1,2}(\bOm)\}\\
A_\delta U=-F.
\end{cases}
\end{equation}

The following characterization of the operator $A_\delta$ can be found in \cite{War}.

\begin{equation}
\begin{cases}
D(A_\delta)=\{U:=(u,u_{\Gamma})\in \mathbb W_\delta^{1,2}(\bOm),\; \Delta u\in L^2(\Omega),\;\delta\Delta_\Gamma u_{\Gamma}-\partial_\nu u\;\mbox{ exists in }\; L^2(\Gamma)\}\\
A_\delta U=\Big(\gamma\Delta u,\delta\Delta_\Gamma u_{\Gamma}-\gamma\partial_\nu u-\beta u_{\Gamma}\Big),
\end{cases}
\end{equation}
that is,  on its domain, $A_\delta$ is the matrix operator
\[
A_\delta=\left(
\begin{matrix}
\gamma \Delta &0\\
-\gamma\partial_\nu &\delta\Delta_\Gamma-\beta
\end{matrix}
\right).
\]

Throughout the remainder of the article for a function $F=(f,g)\in\mathbb X^2(\bOm)$, by $F\ge 0$, we mean that $f\ge 0$ a.e. in $\Omega$ and $g\ge 0$ $\sigma$-a.e. on $\Gamma$. We shall also denote $F^+=(f^+,g^+)$ and $F^-=(f^-,g^-)$ where $f^+=\sup\{f,0\}$ and $f^-=\sup\{-f,0\}$.
We have the following result.

\begin{proposition}\label{pro-sg}
The operator $A_\delta$ generates a strongly continuous  analytic semigroup $(e^{tA_\delta})_{t\ge 0}$ on $\mathbb X^2(\bOm)$ which is also submarkovian. That is, the semigroup is positive and contractive on $\mathbb X^\infty(\bOm)$.
\end{proposition}

\begin{proof}
Since the symmetric form $\mathcal E_\delta$ is closed, continuous, elliptic and $\mathbb W_\delta^{1,2}(\bOm)$ is dense in $\mathbb X^2(\bOm)$,  we have that $A_\delta$ generates a strongly continuous and analytic semigroup $(e^{tA_\delta})_{t\ge 0}$ on $\mathbb X^2(\bOm)$.  Next we show that the semigroup is positive. Let $U=(u,u_\Gamma)\in \mathbb W_\delta^{1,2}(\bOm)$. Then $U^+=(u^+,u_\Gamma^+)\in \mathbb W_\delta^{1,2}(\bOm)$ and a simple calculation gives $\mathcal E_\delta(U^+,U^-)=0$. By \cite[Theorem 1.3.2]{Dav}, this implies that the semigroup is positive. Next, for $0\le U=(u,u_\Gamma)\in \mathbb W_\delta^{1,2}(\bOm)$, we let $U\wedge 1=(u\wedge 1,u_\Gamma\wedge 1)$. It is also easy to see that for every $0\le U\in \mathbb W_\delta^{1,2}(\bOm)$ we have that $U\wedge 1\in \mathbb W_\delta^{1,2}(\bOm)$ and a simple calculation gives $\mathcal E_\delta(U\wedge 1,U\wedge 1)\le\mathcal E_\delta(U,U)$.
By \cite[Theorem 1.3.3]{Dav}, this implies that the semigroup is contractive on $\mathbb X^\infty(\bOm)$. We have shown that the semigroup is submarkovain and the proof is finished.
\end{proof}

We adopt the following notion of solutions to the system \eqref{heat_dbc}.

\begin{definition}\label{def-sol}
Let $g\in L^2(\Sigma_T)$, $F:=(0,g)$ and $U_0:=(u_0,u_{0,\Gamma})\in\mathbb X^2(\bOm)$.
\begin{enumerate}
\item A function $u$ is said to be a strong solution of \eqref{heat_dbc} if $U:=(u,u|_{\Gamma})\in W^{1,2}((0,T);\mathbb X^2(\bOm))\cap L^2((0,T);D(A_\delta))$ and fulfills \eqref{heat_dbc}.

\item A function $u$ is called a mild solution of \eqref{heat_dbc} if $U:=(u,u|_{\Gamma})\in C([0,T];\mathbb X^2(\bOm))$ and satisfies
\begin{align}\label{mild-sol}
U(\cdot, t)=e^{tA_\delta}U_0+\int_0^te^{(t-s)A_\delta}F(\cdot,s)\;ds\;\mbox{ in }\;\mathbb X^2(\bOm),\;\;t\in [0,T].
\end{align}
%\item A function $u$ is said to be a distributional solution of \eqref{heat_dbc} if $U:=(u,u|_{\Gamma})\in L^2((0,T);\mathbb X^2(\bOm))$ and for every $\Phi:=(\varphi,\varphi|_{\Gamma})\in W^{1,2}((0,T);\mathbb X^2(\bOm))\cap L^2((0,T);D(A_\delta))$ with $\Phi(T,\cdot)=0$ we have
%\begin{align}
%&\int_{\Omega_T}(-\partial_t\varphi-\gamma\Delta\varphi)u\;dxdt+\int_{\Sigma_T}(-\partial_t\varphi-\delta\Delta_\Gamma%\varphi+\gamma\partial_\nu \varphi+\beta\varphi)u\;d\sigma dt\notag\\
%&=\int_{\Sigma_T}g\varphi\;d\sigma dt +\int_{\Omega}u_0\varphi(0,x)\;dx+\int_{\Gamma}u_{0,\Gamma}\varphi(0,x)\;d\sigma.
%\end{align}
\end{enumerate}
\end{definition}

Next, using a simple integration by parts, we have that the adjoint system associated to  \eqref{heat_dbc} is  given by the following backward problem

\begin{align}\label{heat_dbc_adj}
\begin{cases}
\partial_t\phi + \gamma\Delta\phi = 0 \;\;\;&\mbox{ in}\;\;\Omega\times(0,T)=\Omega_T\\
\partial_t\phi_{\Gamma}+\delta\LB\phi_\Gamma  - \gamma\partial_{\nu}\phi - \beta\phi_\Gamma = 0 \;\;\;&\mbox{ on }\;\;\Gamma\times(0,T)=\Sigma_T\\
\left.(\phi,\phi_{\Gamma})\right|_{t=T} = (\phi_T,\phi_{\Gamma,T})\;\;&\mbox{ in }\;  \Omega\times\Gamma.
\end{cases}
\end{align}
Here too, $\phi_{\Gamma,T}$ is not a priori the trace of  $\phi_{T}$ since we did not assume that $\phi_T$ has a trace, but if $\phi_T$ has a well defined trace then the trace must coincide with $\phi_{\Gamma,T}$.

We notice that using the operator $A_\delta$, we have that the system \eqref{heat_dbc} can be rewritten as an abstract Cauchy problem
\begin{equation}\label{ACP1}
\partial_tU-A_\delta U=F\;\;\mbox{ in }\;\Omega_T\times\Sigma_T,\;\;\;U(\cdot,0)=(u_0,u_{\Gamma,0})\;\mbox{ on }\;\Omega\times\Gamma,
\end{equation}
where $U:=(u,u_\Gamma)$ and $F:=(0,g)$. Similarly, we have that the system \eqref{heat_dbc_adj} can be rewritten as
\begin{align}\label{ACP2}
\partial_t\Phi+A_\delta \Phi=0\;\;\;\mbox{ in }\;\Omega_T\times\Sigma_T,\;\;\;\Phi(\cdot,T)=(\phi_T,\phi_{\Gamma,T})\;\mbox{ on }\;\Omega\times\Gamma,
\end{align}
where $\Phi:=(\phi,\phi_\Gamma)$.

We have the following result of existence and uniqueness of solutions as a direct consequence of the generation result in Proposition \ref{pro-sg}.

\begin{proposition}\label{ex-sol}
The following assertions hold.
\begin{enumerate}
\item For every $U_0:=(u_0,u_{\Gamma,0})\in\mathbb X^2(\bOm)$ and $g\in L^2(\Sigma_T)$, the Cauchy problem \eqref{ACP1}, and hence the system \eqref{heat_dbc}, has a unique mild solution $U$ given by \eqref{mild-sol}. Moreover, there exists a constant $C>0$ such that
\begin{align}\label{est-msol}
\|U\|_{C([0,T];\mathbb X^2(\bOm))}\le C\left(\|U_0\|_{\mathbb X^2(\bOm)}+\|g\|_{L^2(\Sigma_T)}\right).
\end{align}

\item For every $U_0:=(u_0,u_{\Gamma,0})\in\mathbb W_\delta^{1,2}(\bOm)$ and $g\in L^2(\Sigma_T)$, the Cauchy problem \eqref{ACP1}, and hence the system \eqref{heat_dbc}, has a unique strong solution $U$ and there is a constant $C>0$ such that
\begin{align}
\|U\|_{W^{1,2}((0,T);\mathbb X^2(\bOm))\cap L^2((0,T);D(A_\delta))}\le C\left(\|U_0\|_{\mathbb W_\delta^{1,2}(\bOm)}+\|g\|_{L^2(\Sigma_T)}\right).
\end{align}

\item For every $\Phi_T:=(\phi_T,\phi_{\Gamma,T})\in \mathbb X^2(\bOm)$ {\em (resp. $\Phi_T\in\mathbb W_\delta^{1,2}(\bOm)$)} the Cauchy problem \eqref{ACP2}, and hence the backward system \eqref{heat_dbc_adj}, has a unique mild solution $\Phi$ {\em(resp. strong solution)} given by 
\begin{align*}
\Phi(\cdot,t)=e^{(T-t)A_\delta}\Phi_T\;\mbox{ in }\; \mathbb X^2(\bOm),\;t\in [0,T].
\end{align*}
\end{enumerate}
\end{proposition}

The generation of semigroup given in Proposition \ref{pro-sg} and the proof of the existence and regularity of mild and strong  solutions stated in Proposition \ref{ex-sol} can be done by using the general well-posedness results of Cauchy problems associated with maximal monotone operators contained in the monograph by Brezis \cite{Bre}. We also mention that Proposition \ref{ex-sol} has been also completely proved in \cite{MS2013} by assuming that $\Omega$ is smooth. 

%We conclude this section with the following remark.

\begin{remark}
{\em It is easy to see that every strong solution is also a mild solution and a mild solution satisfying the regularity given in Definition \ref{def-sol}(a) is also a strong solution. For more detail we refer to \cite{Bre,MS2013} and their references.}
\end{remark}

\subsection{The main result}\label{sub-mr}

In this (sub)section we state the results concerning the null controllability of the system \eqref{heat_dbc}. We start with a necessary and sufficient condition for the system to be null controllable.

\begin{proposition}
The following assertions are equivalent.
\begin{enumerate}
\item[(i)] The system \eqref{heat_dbc} is null controllable in time $T>0$.

\item[(ii)] For every $(u_0,u_{\Gamma,0})\in\mathbb X^2(\bOm)$, there exists a control function $g\in L^2(\Sigma_T)$ such that
\begin{align}\label{CFT}
-\int_0^T\int_{\Gamma}g(x,t)\phi_{\Gamma}(x,t)\;d\sigma dt=\int_{\Omega}u_0(x)\phi(x,0)\;dx+\int_{\Gamma}u_{\Gamma,0}(x)\phi_{\Gamma}(x,0)\;d\sigma,
\end{align}
for every $(\phi_T,\phi_{\Gamma,T})\in \mathbb X^2(\bOm)$, where $\phi$ is the unique mild solution of the backward system \eqref{heat_dbc_adj} with final data $(\phi_T,\phi_{\Gamma,T})$. 
\end{enumerate}
\end{proposition}

\begin{proof}
Let $g\in L^2(\Sigma_T)$ be arbitrary and $u$ the solution of \eqref{heat_dbc}. Let $(\phi_T,\phi_{\Gamma,T})\in \mathbb W_\delta^{1,2}(\bOm)$ and $\phi$ the strong solution of \eqref{heat_dbc_adj}. Multiplying the first equation in \eqref{heat_dbc} by $\phi$ and integrating by parts we get that
\begin{align}\label{IT}
0=&\int_0^T\int_{\Omega}\Big(\partial_tu-\gamma\Delta u\Big)\phi\;dxdt=\int_{\Omega}\Big(u(x,T)\phi_T(x)-u_0(x)\phi(x,0)\Big)\;dx-\int_0^T\int_{\Omega}u\partial_t\phi\;dxdt\notag\\
&-\int_0^T\int_{\Omega}\gamma u\Delta\phi\;dxdt-\gamma\int_0^T\int_{\Gamma}\partial_{\nu}u\phi\;d\sigma dt+\gamma\int_0^T\int_{\Gamma}\partial_{\nu}\phi u\;d\sigma dt.
\end{align}
Multiplying the second equation in \eqref{heat_dbc} by  $\phi_\Gamma$ and integrating by parts we get that
\begin{align}\label{BI}
\int_0^T\int_{\Gamma}g(x,t)\phi_\Gamma(x,t)\;d\sigma dt=&\int_{\Gamma}\Big(u_\Gamma(x,T)\phi_{\Gamma,T}(x)-u_\Gamma(x,0)\phi_\Gamma(x,0)\Big)\;d\sigma-\int_0^T\int_{\Gamma}u_\Gamma\partial_t\phi_\Gamma\;d\sigma dt\notag\\
&-\int_0^T\int_\Gamma u_\Gamma\delta\Delta_\Gamma\phi_\Gamma\;d\sigma dt+\int_0^T\int_{\Gamma}\Big(\gamma\partial_\nu u+\beta u_\Gamma\Big)\phi_\Gamma\;d\sigma dt.
\end{align}
Adding \eqref{IT} and \eqref{BI} and using that $\phi$ is the solution of \eqref{heat_dbc_adj} we get that
\begin{align}\label{IE}
\int_0^T\int_{\Gamma}g(x,t)\phi_\Gamma(x,t)\;d\sigma dt=&\int_{\Omega}\Big(u(x,T)\phi_T(x)-u_0(x)\phi(x,0)\Big)\;dx\notag\\
&+\int_{\Gamma}\Big(u_\Gamma(x,T)\phi_{\Gamma,T}(x)-u_{\Gamma,0}(x)\phi_\Gamma(x,0)\Big)\;d\sigma\notag\\
=&-\int_{\Omega}u_0(x)\phi(x,0)-\int_{\Gamma}u_{\Gamma,0}(x)\phi_\Gamma(x,0)\;d\sigma\notag\\
&+\int_{\Omega}u(x,T)\phi_T(x)\;dx+\int_{\Gamma}u_{\Gamma}(x,T)\phi_{\Gamma,T}(x)\;d\sigma.
\end{align}
By approximation, we have that the identity \eqref{IE} also holds for every $(\phi_T,\phi_{\Gamma,T})\in\mathbb X^2(\bOm)$ and $\phi$ the mild solution of \eqref{heat_dbc_adj}.

Now assume that the system \eqref{heat_dbc} is null controllable in time $T>0$. Then by definition, $u(\cdot,T)=0$ in $\Omega$ and $u_\Gamma(\cdot,T)=0$ on $\Gamma$. Hence, it follows from \eqref{IE} that the identity \eqref{CFT} holds and we have shown that (i) implies (ii).  
Finally assume that \eqref{CFT} holds. Then it  follows from \eqref{IE} again that
\begin{align*}
\int_{\Omega}u(x,T)\phi_T(x)\;dx+\int_{\Gamma}u_{\Gamma}(x,T)\phi_{\Gamma,T}(x)\;d\sigma=\langle (u(\cdot,T),u_{\Gamma}(\cdot,T));(\phi_T,\phi_{\Gamma,T})\rangle_{\mathbb X^2(\bOm)}=0
\end{align*}
for every $(\phi_T,\phi_{\Gamma,T})\in\mathbb X^2(\bOm)$. Hence, $u(\cdot,T)=0$ in $\Omega$, $u_\Gamma(\cdot,T)=0$ on $\Gamma$ and we have shown that the system \eqref{heat_dbc} is null controllable in time $T>0$. The proof is finished.
\end{proof}

The following theorem is the main result of the paper.

\begin{theorem}\label{dbc_control_thm}
For every $T>0$ and $(u_0,u_{\Gamma,0})\in\mathbb X^2(\bOm)$, there exists a control function $g\in L^2(\Sigma_T)$ such that the unique mild solution $u$ of \eqref{heat_dbc} satisfies $u(\cdot,T)=0$ in $\Omega$ and $u_\Gamma(\cdot,T)=0$ on $\Gamma$. 
\end{theorem}

\section{Proof of the main result}

In this section we give the proof of Theorem \ref{dbc_control_thm}.
To do this we need some important intermediate results. We start with the  so called Carleman estimate.

\subsection{The Carleman  estimates}

The following theorem gives a suitable Carleman type estimate for solutions of the backward system \eqref{heat_dbc_adj}.
It is one of the main tool needed in the proof of the main result. The proof of the Carleman estimates given here uses some ideas of the corresponding result in the case of the interior null controllability of the system studied in \cite{MS2013}. The weight functions used are the same as the ones in \cite{FG} for the Dirichlet boundary condition.

\begin{theorem}\label{carleman_thm}
Let $T>0$ and $m>1$ be real numbers. Given a positive parameter $\lambda$, we define the weight function $\alpha$ on $\bOm\times(0,T)$ by  
\begin{align}\label{weight_alpha}
\alpha(x,t) = \theta(t)p(x) := \frac{1}{t(T-t)}\left(e^{2\lambda m\norm{\eta}{\infty}}-e^{\lambda\left(m\norm{\eta}{\infty}+\eta(x)\right)}\right),
\end{align}
where $\eta\in C^2(\overline{\Omega})$ is  such that $\eta>0$ in $\Omega$ and $\eta=0$ on $\Gamma$. 
Then, there exists a constant $C>0$ and $\lambda_0,R_0>1$ such that for all $\lambda>\lambda_0$ and $R>R_0$ the strong solution $\phi$ of the backward system \eqref{heat_dbc_adj} satisfies  
\begin{align}\label{carleman}
& \lambda^3R^2   \qint \theta^3\xi^3e^{-2R\alpha}\phi^2\,dxdt + \lambda \qint \theta\xi e^{-2R\alpha}|\nabla \phi|^2\,dxdt \notag\\
&+ \lambda^2 R^2 \sint\theta^3\xi^3e^{-2R\alpha}\phi_{\Gamma}^2\,d\sigma dt \leq C\int_{\Sigma_T} \theta\xi e^{-2R\alpha}\left|\partial_t\phi_{\Gamma} + \delta\LB \phi_\Gamma - \gamma\partial_{\nu}\phi\right|^2\,d\sigma dt,
\end{align}
where, for simplicity of the notation, we have set
\begin{align}\label{weight_xi}
	\xi(x) := e^{\lambda\left(m\norm{\eta}{\infty}+\eta(x)\right)}.
\end{align} 
\end{theorem}

\begin{proof} 
We give the proof in several steps. Throughout the proof $C$ will denote a generic constant which is independent of $\lambda$, $R$ and $\phi$. This constant may vary even from line to line.\\

\noindent
{\bf Step 1. An auxiliary problem.} For any strong solution $\phi$ of the adjoint system \eqref{heat_dbc_adj} and for any fixed $R>0$, we define
\begin{align}\label{zeta}
	z(x,t):=\phi(x,t)e^{-R\alpha(x,t)}.
\end{align}
We notice that
\begin{equation}\label{bc time}
\begin{cases}
	z(\cdot,0):=\lim_{t\to 0}z(\cdot,t)=0=z(\cdot,T):=\lim_{t\to T}z(\cdot,t)\;\;&\mbox{ in }\;\Omega\\
	z_\Gamma(\cdot,0):=\lim_{t\to 0}z_\Gamma(\cdot,t)=0=z_\Gamma(\cdot,T):=\lim_{t\to T}z_\Gamma(\cdot,t)\;\;&\mbox{ on }\;\Gamma
	\end{cases}
\end{equation}
The parameter $R$ is meant to be large. Plugging the function $\phi(x,t)=z(x,t)e^{R\alpha(x,t)}$ in the system \eqref{heat_dbc_adj} and using \eqref{bc time}, we obtain that $z$ satisfies the following system

\begin{align}\label{heat_dbc_z}
\begin{cases}
	\partial_tz + \gamma\Delta z + 2R\gamma\nabla\alpha\cdot\nabla z + (R\gamma \Delta\alpha + R\alpha_t) z + R^2|\nabla\alpha|^2 z=0 \;\;&\mbox{ in }\;  \Omega_T\\
\partial_t	z_\Gamma + \delta\LB z_\Gamma - \gamma\partial_{\nu}z -\left(\beta-R\alpha_t+R\gamma\partial_{\nu}\alpha\right)z_\Gamma = 0 &\mbox{ on }\;\Sigma_T\\
	z(\cdot,0)=z(\cdot,T)=0 &\mbox{ in }\;\Omega\\
		z_\Gamma(\cdot,0)=z_\Gamma(\cdot,T)=0 &\mbox{ on }\;\Gamma
\end{cases}
\end{align}
where $\alpha_t:=\partial_t\alpha$.
Next, expanding the spatial derivatives of $\alpha$ by using the chain rule, we obtain 
\begin{align*}
	\nabla\alpha = -\lambda\theta\xi\nabla\eta \;\mbox{ and }\;\;	\Delta\alpha =  -\lambda\theta\xi\Delta\eta - \lambda^2\theta\xi|\nabla\eta|^2, 
\end{align*}
where we recall that we have used the notation \eqref{weight_xi}.
Using the above expressions, the system \eqref{heat_dbc_z} can be rewritten as 
\begin{align}\label{eq_abbr}
\begin{cases}
	M_1 + M_2 = f\;\;&\textrm{ in } \Omega_T\\
	 N_1 + N_2 = h\;\;&\textrm{ on }\Sigma_T
\end{cases}
\end{align}
where
\begin{align*}
\begin{cases}
	 M_1& := -2\lambda^2R\gamma\theta\xi|\nabla\eta|^2z - 2\lambda R\gamma\theta\xi\left(\nabla\eta\cdot\nabla z\right) + \partial_tz := M_{1,1}+M_{1,2}+M_{1,3}\\
M_2 &:= \lambda^2R^2\gamma\theta^2\xi^2|\nabla\eta|^2z + \gamma\Delta z + R\alpha_tz := M_{2,1}+M_{2,2}+M_{2,3}\\
N_1& := \partial_tz_\Gamma + \lambda R\gamma\theta\xi\partial_{\nu}\eta z_\Gamma := N_{1,1}+N_{1,2}\\
N_2 &:= \delta\LB z_\Gamma + R\alpha_t z_\Gamma - \gamma\partial_{\nu}z: = N_{2,1}+N_{2,2}+N_{2,3}\\
f &:= \lambda R\gamma\theta\xi\Delta\eta z  - \lambda^2R\gamma\theta\xi|\nabla\eta|^2 z\\
 h &:= \beta z_\Gamma.
	\end{cases}
\end{align*}
Applying the respective $L^2$-norms to the terms in \eqref{eq_abbr} we get that
\begin{align}\label{eq_norm}
	\norm{f}{L^2(\Omega_T)}^2 + \norm{h}{L^2(\Sigma_T)}^2 =& \norm{M_1}{L^2(\Omega_T)}^2 + \norm{M_2}{L^2(\Omega_T)}^2 + \norm{N_1}{L^2(\Sigma_T)}^2 + \norm{N_2}{L^2(\Sigma_T)}^2 \notag\\
	&+ 2\langle M_1,M_2\rangle_{L^2(\Omega_T)} + 2\langle N_1,N_2\rangle_{L^2(\Sigma_T)},
\end{align}
which clearly implies that
\begin{align}\label{ineq_norm}
	2\langle M_1, & M_2\rangle_{L^2(\Omega_T)} - \norm{f}{L^2(\Omega_T)}^2 + 2\langle N_1,N_2\rangle_{L^2(\Sigma_T)} - \norm{h}{L^2(\Sigma_T)}^2 \notag\\ 
=& -\sum_{i=1}^2\left(\norm{M_i}{L^2(\Omega_T)}^2 + \norm{N_i}{L^2(\Sigma_T)}^2\right)\leq 0.
\end{align}

\noindent
{\bf Step 2. Estimate from below of the terms defined on $\Omega_T$ in \eqref{ineq_norm}}.
We compute and estimate from below the scalar product $\langle M_1,M_2\rangle_{L^2(\Omega_T)}-\|f\|_{L^2(\Omega_T)}^2$.

 {\bf Step 2.1. Estimate from below of $\langle M_1,M_{2,1}\rangle_{L^2(\Omega_T)}$}. Calculating we  have that
\begin{align}\label{M1}
	\langle M_{1,1},M_{2,1}\rangle_{L^2(\Omega_T)} &= -2\lambda^4R^3\gamma^2 \qint \theta^3\xi^3|\nabla\eta|^4 z^2\,dxdt.
\end{align}
A simple calculation and an integrating by parts yield
\begin{align}\label{M2}
	\langle M_{1,2}, & M_{2,1}\rangle_{L^2(\Omega_T)} \notag\\
=& -\lambda^3R^3\gamma^2 \qint \theta^3\xi^3|\nabla\eta|^2\left(\nabla\eta\cdot\nabla (z^2)\right)\,dxdt \notag\\
	= &- \lambda^3R^3\gamma^2  \int_{\Sigma_T}  \theta^3\xi^3|\nabla\eta|^2\partial_{\nu}\eta z^2\,d\sigma dt + \lambda^3R^3\gamma^2 \qint \theta^3\textrm{div}(\xi^3|\nabla\eta|^2\nabla\eta)z^2\,dxdt\notag\\
	=& \mathbf{BT}_1 + 3\lambda^4R^3\gamma^2 \qint \theta^3\xi^3|\nabla\eta|^4 z^2\,dxdt + \lambda^3R^3\gamma^2 \qint \theta^3\xi^3\Delta\eta|\nabla\eta|^2 z^2\,dxdt \notag\\
	&+ \lambda^3R^3\gamma^2 \qint \theta^3\xi^3\left(\nabla\eta\cdot\nabla(|\nabla\eta|^2)\right) z^2\,dxdt,
\end{align}
where we have set the boundary term
\begin{align*}
 \mathbf{BT}_1 :=- \lambda^3R^3\gamma^2 \int_{\Sigma_T} \theta^3\xi^3|\nabla\eta|^2\partial_{\nu}\eta z^2\,d\sigma dt.
\end{align*}
Adding \eqref{M1} and \eqref{M2} and using that 
\begin{align*}
\Delta\eta\ge -|\Delta\eta|\;\mbox{ and }\;\nabla \eta\cdot\nabla(|\nabla\eta|^2)\ge -|\nabla \eta||\nabla(|\nabla\eta|^2)|,
\end{align*}
we get that 
\begin{align}\label{M11+12-1}
	&\langle M_{1,1}+M_{1,2},M_{2,1}\rangle_{L^2(\Omega_T)}\notag\\
	 \geq &\mathbf{BT}_1 + \lambda^4R^3\gamma^2 \qint \theta^3\xi^3\left(|\nabla\eta|^4 - \frac{1}{\lambda}|\Delta\eta|\,|\nabla\eta|^2\right) z^2\,dxdt\notag\\
	&- \lambda^3R^3\gamma^2 \qint \theta^3\xi^3|\nabla\eta||\nabla(|\nabla\eta|^2)| z^2\,dxdt\notag\\
=&\mathbf{BT}_1+ \lambda^4R^3\gamma^2 \qint \theta^3\xi^3\left(|\nabla\eta|^4 - \frac{1}{\lambda}|\Delta\eta|\,|\nabla\eta|^2- \frac 1\lambda|\nabla\eta||\nabla(|\nabla\eta|^2)|\right) z^2\,dxdt.
\end{align}
Let
\begin{align}\label{lambda}
	\lambda\geq\max\left\{\frac{2|\Delta\eta|}{|\nabla\eta|^2},\frac{4|\nabla(|\nabla\eta|^2)|}{|\nabla\eta|^3}\right\}.
\end{align}
Then
\begin{align}\label{eta-2}
|\nabla\eta|^4 - \frac{1}{\lambda}|\Delta\eta|\,|\nabla\eta|^2\geq \frac 12|\nabla\eta|^4\;\;\;\mbox{ and }\;\;\;\frac 12|\nabla\eta|^4 - \frac 1\lambda|\nabla\eta||\nabla(|\nabla\eta|^2)|\geq \frac 14|\nabla \eta|^4.
\end{align}
It follows from \eqref{eta-2} that if $\lambda$ satisfies \eqref{lambda}, then
\begin{align}\label{eta}
|\nabla\eta|^4 - \frac{1}{\lambda}|\Delta\eta|\,|\nabla\eta|^2- \frac 1\lambda|\nabla\eta||\nabla(|\nabla\eta|^2)|\geq \frac 14|\nabla \eta|^4.
\end{align}
Using \eqref{eta} we get from \eqref{M11+12-1} that there exists a constant $C>0$ (depending only on $\eta$ and $\gamma$) such that if $\lambda$ satisfies \eqref{lambda}, then
\begin{align}\label{M11+12}
	\langle M_{1,1}+M_{1,2},M_{2,1}\rangle_{L^2(\Omega_T)} \geq  \mathbf{BT}_1 + \lambda^4R^3C \qint \theta^3\xi^3 z^2\,dxdt.
\end{align}
Moreover, we have that there is a constant $C>0$ (depending only on $\eta$ and $\gamma$) such that
\begin{align}\label{M13-21}
	\langle M_{1,3},M_{2,1}\rangle_{L^2(\Omega_T)} &= -\lambda^2R^2\gamma \qint \theta\theta_t\xi^2|\nabla\eta|^2z^2\,dxdt \geq -\lambda^2R^2C \qint \theta^3\xi^3z^2\,dxdt,
\end{align}
where we have used the fact that there is a constant $C>0$ such that $|\theta_t|\le C\theta^2$. It follows from \eqref{M11+12} and \eqref{M13-21} that if $\lambda$ satisfies \eqref{lambda}, then
\begin{align}\label{M1-21-1}
\langle M_1,M_{2,1}\rangle_{L^2(\Omega_T)} \ge \mathbf{BT}_1 + \lambda^4R^3C \qint \theta^3\xi^3\Big(1-\frac{1}{R\lambda^2}\Big) z^2\,dxdt.
\end{align}
If $\lambda^2\ge \frac{2}{R}$, then $\left(1-\frac{1}{R\lambda^2}\right)\ge\frac 12$. Thus it follows from \eqref{M1-21-1} that there exists a constant $C>0$ such that if
\begin{align*}
	\lambda\geq\max\left\{\frac{2|\Delta\eta|}{|\nabla\eta|^2},\frac{4|\nabla(|\nabla\eta|)|}{|\nabla\eta|^3},\sqrt{\frac{2}{R}}\;\right\},
\end{align*}
then
\begin{align}\label{M1-21}
	\langle M_1,M_{2,1}\rangle_{L^2(\Omega_T)} &\geq \mathbf{BT}_1 + \lambda^4R^3C \qint \theta^3\xi^3 z^2\,dxdt.
\end{align}

{\bf Step 2.2. Estimate from below of $\langle M_1,M_{2,2}\rangle_{L^2(\Omega_T)}$}. Integrating by parts and using the fact that
\begin{align}\label{e-eta}
\nabla \xi=\lambda\xi\nabla\eta,
\end{align}
we  get that
\begin{align}\label{M11}
	\langle M_{1,1},M_{2,2}\rangle_{L^2(\Omega_T)} =& -2\lambda^2R\gamma^2 \sint \theta\xi|\nabla\eta|^2 z\partial_{\nu}z\,d\sigma dt \notag\\
	&+ 2\lambda^2R\gamma^2 \qint \theta\xi|\nabla\eta|^2|\nabla z|^2\,dxdt\notag
	\\
	&+ 2\lambda^2R\gamma^2 \qint \theta\xi\left(\nabla(|\nabla\eta|^2)\cdot\nabla z\right)z\,dxdt  \notag\\
	&+ 2\lambda^3R\gamma^2 \qint \theta\xi|\nabla\eta|^2\left(\nabla\eta\cdot\nabla z\right)z\,dxdt\notag
	\\
	= :&\bt{2} + \dt{1} + \dt{2} + \dt{3}.
\end{align}
Now applying the Young inequality, we get that there exists a constant $C>0$ (depending only on $\eta$ and $\gamma$) such that
\begin{align}\label{dt2}
	|\dt{2}| =&\left|2\lambda^2R\gamma^2 \qint \theta\xi\left(\nabla(|\nabla\eta|^2)\cdot\nabla z\right)z\,dxdt\right|\notag\\
	&\leq \lambda^4RC \qint \theta\xi z^2\,dxdt + RC\qint \theta\xi|\nabla z|^2\,dxdt
	\end{align}
and
\begin{align}\label{dt3}
	|\dt{3}| =&\left|2\lambda^3R\gamma^2 \qint \theta\xi|\nabla\eta|^2\left(\nabla\eta\cdot\nabla z\right)z\,dxdt\right|\notag\\
	&\leq \lambda^4R^2C \qint\theta^2\xi^2z^2\,dxdt + \lambda^2C \qint |\nabla z|^2\,dxdt.
\end{align}
Using \eqref{dt2} and \eqref{dt3}, we get from \eqref{M11} that there exists a constant $C>0$ such that
\begin{align}\label{M11-1}
	\langle M_{1,1},M_{2,2}\rangle_{L^2(\Omega_T)} \geq& \dt{1} - \lambda^4R^2C \qint \left(1+\frac{1}{\lambda^2R\theta\xi}\right)\theta^2\xi^2z^2\,dxdt\notag\\
	& - C \qint \left(R\theta\xi+\lambda^2\right)|\nabla z|^2\,dxdt + \bt{2}.
\end{align}
Furthermore, we have that
\begin{align}\label{M11-2}
	\langle M_{1,2},&M_{2,2}\rangle_{L^2(\Omega_T)} 
	\notag\\ 
	=& -2\lambda R\gamma^2 \sint \theta\xi\left(\nabla\eta\cdot\nabla z\right)\partial_{\nu}z\,d\sigma dt + 2\lambda R\gamma^2 \qint \theta\nabla\left(\xi(\nabla\eta\cdot\nabla z)\right)\cdot\nabla z\,dxdt
\notag	\\
	= &-2\lambda R\gamma^2 \sint \theta\xi\partial_{\nu}\eta|\partial_{\nu}z|^2\,d\sigma dt + 2\lambda R\gamma^2 \qint \theta\xi\nabla(\nabla\eta\cdot\nabla z)\cdot\nabla z\,dxdt 
\notag	\\
	&+ 2\lambda R\gamma^2 \qint \theta(\nabla\eta\cdot\nabla z)(\nabla\eta\cdot\nabla z)\,dxdt
\notag	\\
	= &\bt{3} + 2\lambda R\gamma^2 \qint \theta\xi D^2\eta(\nabla z,\nabla z)\,dxdt + \lambda R\gamma^2 \qint \theta\xi\left(\nabla\eta\cdot\nabla(|\nabla z|^2)\right)\,dxdt
\notag	\\
	&+ 2\lambda^2R\gamma^2 \qint \theta\xi(\nabla\eta\cdot\nabla z)^2\,dxdt
\notag	\\
	=& \bt{3} + \dt{4} + \dt{5} + \dt{6}.
\end{align}
In the above expression, in $\dt{4}$ we have introduced the notation 
\begin{align*}
	D^2\eta(\zeta,\zeta):= \sum_{i,j=1}^N\partial_{x_ix_j}\eta\zeta_i\zeta_j, \;\;\;\forall\zeta\in\RR^N.
\end{align*}
Moreover we have that there exists a constants $C>0$ such that
\begin{align}\label{M11-3}
	\dt{4} &=2\lambda R\gamma^2 \qint \theta\xi D^2\eta(\nabla z,\nabla z)\,dxdt\notag\\
	&\geq -\lambda RC \qint \theta\xi|\nabla z|^2\,dxdt
\end{align}
and
\begin{align}\label{M11-4}
	\dt{5} =& \lambda R\gamma^2 \sint \theta\xi\partial_{\nu}\eta|\nabla z|^2\,d\sigma dt - \lambda R\gamma^2 \qint \theta\,\textrm{div}(\xi\nabla\eta)|\nabla z|^2\,dxdt \notag\\
	=& \bt{4} - \lambda R\gamma^2 \qint \theta\xi\Delta\eta|\nabla z|^2\,dxdt - \lambda R\gamma^2 \qint \theta\left(\nabla\xi\cdot\nabla\eta\right)|\nabla z|^2\,dxdt
	\notag\\
	\geq &\bt{4}  - \lambda RC\qint \theta\xi|\nabla z|^2\,dxdt - \lambda^2 R\gamma^2 \qint \theta\xi|\nabla\eta|^2\,|\nabla z|^2\,dxdt,
\end{align}
where we have also used  \eqref{e-eta}. In addition we have that there exists a constant $C>0$ such that
\begin{align}\label{dt6}
\dt{6}= 2\lambda^2R\gamma^2 \qint \theta\xi(\nabla\eta\cdot\nabla z)^2\,dxdt\ge 2\lambda^2RC \qint \theta\xi|\nabla z|^2\,dxdt.
\end{align}
Since $\nabla z$ vanishes at $t=0$ and at $t=T$, then integrating by parts, we get that
\begin{align}\label{M11-5}
	\langle M_{1,3},M_{2,2}\rangle_{L^2(\Omega_T)} &= \gamma\sint z_t\partial_{\nu}z\,d\sigma dt =: \bt{5}.
\end{align}
Combining \eqref{M11-1}, \eqref{M11-2}, \eqref{M11-3}, \eqref{M11-4}, \eqref{dt6}  and \eqref{M11-5}  we get that for $\lambda$ large enough,
\begin{align}\label{M1-22}
	\langle M_1,M_{2,2}\rangle_{L^2(\Omega_T)} \geq &\bt{1} + \bt{2} + \bt{3} + \bt{4} + \bt{5} + \lambda^2RC \qint \theta\xi|\nabla z|^2\,dxdt\notag \\
&- \lambda^4R^2C \qint \theta^2\xi^2z^2\,dxdt,
\end{align}
for some constant $C>0$ depending only on $\eta$ and $\gamma$.\\

{\bf Step 2.3. Estimate from below of $\langle M_1,M_{2,3}\rangle_{L^2(\Omega_T)}$}. First, we notice that there exist two constants $\varsigma_1>0$ and $\varsigma_2>0$ such that
\begin{align}\label{alpha_t_est}
	|\alpha_t|\leq\varsigma_1\theta^2\xi^2\;\mbox{ and }\;\; |\alpha_{tt}|\leq\varsigma_2\theta^3\xi^3.
\end{align}
Next, using \eqref{alpha_t_est}, we get that there exists a constant $C>0$ such that if $\lambda$ is large enough, then
\begin{align}\label{M1-3}
	\langle M_{1,1},M_{2,3}\rangle_{L^2(\Omega_T)} &= -2\lambda^2R^2\gamma \qint \theta\xi|\nabla\eta|^2\alpha_tz^2\,dxdt \notag
	\\
	&\geq -\lambda^2R^2C\qint \theta^3\xi^3z^2\,dxdt.
\end{align}
Calculating and integrating by parts we get that

\begin{align}\label{M12-1}
	\langle M_{1,2},M_{2,3}&\rangle_{L^2(\Omega_T)}\notag
	\\
	= &-\lambda R^2\gamma \qint \theta\xi\alpha_t\left(\nabla\eta\cdot\nabla(z^2)\right)\,dxdt \notag
	\\
	=& -\lambda R^2\gamma \sint \theta\xi\alpha_t\partial_{\nu}\eta z^2\,d\sigma dt + \lambda R^2\gamma \qint \theta\textrm{div}(\xi\alpha_t\nabla\eta)z^2\,dxdt \notag
	\\
	=& \bt{6} + \lambda R^2\gamma \qint \theta\xi\alpha_t\Delta\eta z^2\,dxdt + \lambda R^2\gamma \qint \theta\xi(\nabla\alpha_t\cdot\nabla\eta)z^2\,dxdt \notag
	\\
	&+ \lambda R^2\gamma \qint \theta\alpha_t(\nabla\xi\cdot\nabla\eta)z^2\,dxdt \notag
	\\
	 = &\bt{6} + \lambda R^2\gamma \qint \theta\xi\alpha_t\Delta\eta z^2\,dxdt + \lambda R^2\gamma \qint \theta\xi(\nabla\alpha_t\cdot\nabla\eta)z^2\,dxdt \notag
	\\
	&+ \lambda^2R^2\gamma \qint \theta\alpha_t(\nabla\eta\cdot\nabla\eta)z^2\,dxdt \notag
	\\
	&\geq \bt{6} - \lambda^2R^2C \qint \theta^3\xi^3z^2\,dxdt,
\end{align}
where we have set
\begin{align*}
\bt{6}:= -\lambda R^2\gamma \sint \theta\xi\alpha_t\partial_{\nu}\eta z^2\,d\sigma dt.
\end{align*}
In addition we have that there exists a constant $C>0$ such that

\begin{align}\label{M13-1}
	\langle M_{1,3},M_{2,3}\rangle_{L^2(\Omega_T)} = \frac{R}{2} \qint \alpha_t(z^2)_t\,dxdt 
= -\frac{R}{2} \qint \alpha_{tt}z^2\,dxdt \geq - RC\qint \theta^3\xi^3 z^2\,dxdt.
\end{align}
Combining \eqref{M1-3}, \eqref{M12-1} and \eqref{M13-1}, we get that there exists a constant $C>0$ such that for $\lambda$ and $R$ large enough, we have
\begin{align}\label{M1-23}
	\langle M_1,M_{2,3}\rangle_{L^2(\Omega_T)} &\geq \bt{6} - \lambda^2R^2C \qint \theta^3\xi^3 z^2\,dxdt.
\end{align}
Finally, it follows from \eqref{M1-21}, \eqref{M1-22} and \eqref{M1-23}  that  there is a constant $C>0$ such that for $\lambda$ and $R$ large enough, we have
\begin{align}\label{M1-M2}
	\langle M_1,M_2\rangle_{L^2(\Omega_T)} \geq &\lambda^4R^3C \qint \theta^3\xi^3 z^2\,dxdt + \lambda^2RC\qint \theta\xi|\nabla z|^2\,dxdt \notag
	\\
	&+ \bt{1} + \bt{2} + \bt{3} + \bt{4} + \bt{5} + \bt{6}.
\end{align}

{\bf Step 2.4. Estimate from below of $	-\norm{f}{L^2(\Omega_T)}^2$}.
We have that there exists a constant $C>0$ such that
\begin{align}\label{f}
	-\norm{f}{L^2(\Omega_T)}^2 &= -\norm{\lambda R\gamma\theta\xi\Delta\eta z  - \lambda^2R\gamma\theta\xi|\nabla\eta|^2z}{L^2(Q)}^2 \notag
	\\
	&\geq -\lambda^2R^2\gamma^2 C\qint \theta^2\xi^2|\Delta\eta|^2z^2\,dxdt  - \lambda^4R^2\gamma^2 C\qint \theta^2\xi^2|\nabla\eta|^4 z^2\,dxdt\notag
	\\
	&\geq -\lambda^2R^2C\qint \theta^2\xi^2z^2\,dxdt - \lambda^4R^2C \qint \theta^2\xi^2z^2\,dxdt.
\end{align}

It follows from \eqref{M1-M2} and \eqref{f}, that for $\lambda$ and $R$ large enough, all these terms can be absorbed and we get that there exists a constant $C>0$ such that
\begin{align}\label{scalar_product_dist}
	2\langle M_1,M_2\rangle_{L^2(\Omega_T)} - \norm{f}{L^2(Q)}^2 \geq &\lambda^4R^3C \qint \theta^3\xi^3 z^2\,dxdt + \lambda^2RC \qint \theta\xi|\nabla z|^2\,dxdt \nonumber
	\\
	&+ \bt{1} + \bt{2} + \bt{3} + \bt{4} + \bt{5} + \bt{6}. 
\end{align}

\noindent
{\bf Step 3. Estimate of the boundary terms}.
Let us now compute the boundary terms. Integrating by parts, we get that
\begin{align*}
	\langle N_{1,1}, N_{2,1}\rangle_{L^2(\Sigma_T)} &= \delta\sint \partial_tz_\Gamma\LB z_\Gamma\,d\sigma dt = -\delta\sint \GB (\partial_tz_\Gamma)\cdot\GB z_\Gamma\,d\sigma dt \\
	&= -\frac{\delta}{2} \sint \partial_t\left(|\GB z_\Gamma|^2\right)\,d\sigma dt = 0,
	\\
	\langle N_{1,2}, N_{2,1}\rangle_{L^2(\Sigma_T)} &= \lambda R\gamma\delta \sint \theta\xi\partial_{\nu}\eta z_\Gamma\LB z_\Gamma\,d\sigma dt \\
	&= - \lambda R\gamma\delta \sint \theta\xi z_\Gamma\GB(\partial{_\nu}\eta)\cdot\GB z_\Gamma \,d\sigma dt - \lambda R\gamma\delta \sint \theta\xi\partial_{\nu}\eta |\GB z_\Gamma|^2\,d\sigma dt
	\\
	\langle N_{1,1}, N_{2,2}\rangle_{L^2(\Sigma_T)} &= \frac{R}{2} \sint \alpha_t\partial_t(z_\Gamma^2)\,d\sigma dt = -\frac{R}{2} \sint \alpha_{tt}z_\Gamma^2\,d\sigma dt
	\\
	\langle N_{1,2}, N_{2,2}\rangle_{L^2(\Sigma_T)} &= \lambda R^2\gamma \sint \theta\xi\alpha_t\partial_{\nu}\eta z_\Gamma^2\,d\sigma dt
	\\
	\langle N_{1,1}, N_{2,3}\rangle_{L^2(\Sigma_T)} &= -\gamma \sint \partial_tz_\Gamma\partial_{\nu}z\,d\sigma dt
	\\
	\langle N_{1,1}, N_{2,3}\rangle_{L^2(\Sigma_T)} &= -\lambda R\gamma \sint \theta\xi\partial_{\nu}\eta z_\Gamma\partial_{\nu}z\,d\sigma dt.
\end{align*}
Therefore, we have that
\begin{align*}
	\sum_{i=1}^6 \bt{i} &+ 2\langle N_1, N_2\rangle_{L^2(\Sigma_T)} -\norm{h}{L^2(\Sigma_T)}^2 
	\\
	=& -\lambda^3R^3\gamma^2 \sint \theta^3\xi^3|\nabla\eta|^2\partial_{\nu}\eta z_\Gamma^2\,d\sigma dt - 2\lambda^2R\gamma^2 \sint \theta\xi|\nabla\eta|^2z_\Gamma\partial_{\nu}z\,d\sigma dt 
	\\
	&+\lambda R^2\gamma \sint \theta\xi\alpha_t\partial_{\nu}\eta z_\Gamma^2\,d\sigma dt - R \sint \alpha_{tt}z_\Gamma^2\,d\sigma dt - \sint \beta^2z_\Gamma^2\,d\sigma dt
	\\
	& -2\lambda R\gamma^2 \sint \theta\xi\partial_{\nu}\eta z_\Gamma\partial_{\nu}z\,d\sigma dt - \gamma\sint \partial_t z_\Gamma\partial_\nu z\,d\sigma dt
	\\
	&- 2\lambda R\gamma^2 \sint \theta\xi\partial_\nu\eta|\partial_\nu z|^2\,d\sigma dt + \lambda R\gamma^2 \sint \theta\xi\partial_\nu\eta|\nabla z|^2\,d\sigma dt
	\\
	&- 2\lambda R\gamma \delta\sint \theta\xi z_\Gamma\GB(\partial{_\nu}\eta)\cdot\GB z_\Gamma \,d\sigma dt - 2\lambda R\gamma \delta\sint \theta\xi\partial_{\nu}\eta |\GB z_\Gamma|^2\,d\sigma dt=: J.
\end{align*}
Moreover,
\begin{align*}
	- 2\lambda R\gamma^2 \sint \theta\xi\partial_\nu\eta|\partial_\nu z|^2\,d\sigma dt 
	&= - 2\lambda R\gamma^2 \sint \theta\xi\partial_\nu\eta(\nu\cdot\nabla z)^2\,d\sigma dt\\
	& \geq - 2\lambda R\gamma^2 \sint \theta\xi\partial_\nu\eta|\nabla z|^2\,d\sigma dt,
\end{align*}
where we notice that
\begin{align*}
|\nabla z|^2|_{\Gamma}=|\GB z|^2+|\partial_{\nu}z|^2.
\end{align*}
Hence, we obtain a first estimate for the expression $J$, namely
\begin{align}\label{JJ-1}
	J &\geq -\lambda^3R^3\gamma^2 \sint \theta^3\xi^3|\nabla\eta|^2\partial_{\nu}\eta z_\Gamma^2\,d\sigma dt - 2\lambda^2R\gamma^2 \sint \theta\xi|\nabla\eta|^2z_\Gamma\partial_{\nu}z\,d\sigma dt \notag
	\\
	&+\lambda R^2\gamma \sint \theta\xi\alpha_t\partial_{\nu}\eta z_\Gamma^2\,d\sigma dt - R \sint \alpha_{tt}z_\Gamma^2\,d\sigma dt -\sint  \beta^2z_\Gamma^2\,d\sigma dt\notag
	\\
	& -2\lambda R\gamma^2 \sint \theta\xi\partial_{\nu}\eta z_\Gamma\partial_{\nu}z\,d\sigma dt - \gamma\sint \partial_t z_\Gamma\partial_\nu z\,d\sigma dt - 2\lambda R\gamma \delta\sint \theta\xi z_\Gamma\GB(\partial{_\nu}\eta)\cdot\GB z_\Gamma \,d\sigma dt\notag
	\\
	& - 2\lambda R\gamma \delta\sint \theta\xi\partial_{\nu}\eta |\nabla z|^2\,d\sigma dt -\lambda R\gamma^2 \sint \theta\xi\partial_\nu\eta|\nabla z|^2\,d\sigma dt.
\end{align}
We mention that
\begin{align}\label{cond-eta}
\GB\eta=0,\;\;\;|\nabla\eta|=|\partial_\nu\eta|,\;\;\partial_\nu\eta\le -C<0\;\;\mbox{ on }\;\Gamma,
\end{align}
for some constant $C>0$.
Now since $\left.\partial_{\nu}\eta\right|_{\Gamma} <0$ (by \eqref{cond-eta}),  we have that the last two terms in the right-hand side of \eqref{JJ-1} are positive and we can ignore them. Moreover, we recall that  there exists a constant $C>0$ such that $|\nabla\eta|\leq C$ and that $|\alpha_t|\leq C\theta^2\xi^2$. Therefore, from \eqref{JJ-1} we obtain
\begin{align}\label{JJ-2}
	J \geq &\lambda^3R^3C \sint \theta^3\xi^3 \left(1-\frac{1}{\lambda^2 R}\right) z_\Gamma^2\,d\sigma dt - \lambda^2RC \sint \theta\xi z_\Gamma\partial_{\nu}z\,d\sigma dt - R \sint \alpha_{tt}z_\Gamma^2\,d\sigma dt \notag
	\\
	&- \sint \theta^3\xi^3(\theta\xi)^{-3} \beta^2z_\Gamma^2\,d\sigma dt -2\lambda RC \sint \theta\xi z_\Gamma\partial_{\nu}z\,d\sigma dt - \gamma\sint \partial_t z_\Gamma\partial_\nu z\,d\sigma dt \notag
	\\
	&- 2\lambda R\gamma \delta\sint \theta\xi z_\Gamma\GB(\partial{_\nu}\eta)\cdot\GB z_\Gamma \,d\sigma dt.
\end{align}
In order to treat the last integral in the right-hand side of \eqref{JJ-2}, we recall from Section 1.2 that $\norm{\cdot}{L^2(\Gamma)}+\norm{\GB\cdot}{L^2(\Gamma)}$ defines an equivalent norm on $W^{1,2}(\Gamma)$. Hence, the interpolation inequality \eqref{GB_interpol} yields
\begin{align*}
	\norm{\GB z_\Gamma}{L^2(\Gamma)}^2\leq C \norm{z_\Gamma}{L^2(\Gamma)}\norm{z_\Gamma}{D(\LB)}.
\end{align*}
Therefore, we have
\begin{align}\label{in-delta}
	\left|\sint \delta\theta\xi z_\Gamma\GB(\partial{_\nu}\eta)\cdot\GB z_\Gamma \,d\sigma dt\,\right| &\leq \sint \theta\xi |\GB(\partial{_\nu}\eta)|\,|\delta\GB z_\Gamma| |z_\Gamma|\,d\sigma dt \leq C\sint \theta\xi |\delta\GB z_\Gamma| |z_\Gamma|\,d\sigma dt\notag
	\\
	&\leq C\int_0^T \theta\xi\delta^2\norm{\GB z_\Gamma}{L^2(\Gamma)}^2\,dt + C\sint \theta\xi z_\Gamma^2\,d\sigma dt\notag
	\\
	&\leq C\int_0^T \theta\xi\delta^2\norm{z_\Gamma}{D(\LB)}^2\,dt + C\sint \theta\xi z_\Gamma^2\,d\sigma dt\notag
	\\
	&\leq C\sint \theta\xi |\delta\LB z_\Gamma|^2\,d\sigma dt + C\sint \theta\xi z_\Gamma^2\,d\sigma dt.
\end{align}
Plugging \eqref{in-delta} in \eqref{JJ-2} we get
\begin{align}\label{JJ-3}
	J \geq &\lambda^3R^3C \sint \theta^3\xi^3 \left(1-\frac{1}{\lambda^2 R}\right) z_\Gamma^2\,d\sigma dt - \lambda^2RC \sint \theta\xi z_\Gamma\partial_{\nu}z\,d\sigma dt - R \sint \alpha_{tt}z_\Gamma^2\,d\sigma dt \notag
	\\
	&- \sint \theta^3\xi^3(\theta\xi)^{-3} \beta^2z_\Gamma^2\,d\sigma dt -2\lambda RC \sint \theta\xi z_\Gamma\partial_{\nu}z\,d\sigma dt - \gamma\sint \partial_t z_\Gamma\partial_\nu z\,d\sigma dt \notag
	\\
	&- \lambda RC \sint \theta\xi |\delta\LB z_\Gamma|^2\,d\sigma dt - \lambda RC \sint \theta^3\xi^3 (\theta\xi)^{-2} z_\Gamma^2\,d\sigma dt.
\end{align}
Moreover, by definition of $\theta$ and $\xi$ we also have that there exists a constant $C>0$ such that for all $k\in\NN$, 
\begin{align}\label{theta_xi_est}
	|\theta\xi|^{-k}\leq C.
\end{align}
Thus using \eqref{theta_xi_est}, we get from \eqref{JJ-3} that
\begin{align}\label{JJ-4}
	J \geq &\lambda^3R^3C \sint \theta^3\xi^3 \left(1-\frac{1}{\lambda^2 R}-\frac{1}{\lambda^3 R^3}-\frac{1}{\lambda^2 R^2}\right) z_\Gamma^2\,d\sigma dt - \lambda^2RC \sint \theta\xi z_\Gamma\partial_{\nu}z\,d\sigma dt \notag
	\\
	&- R \sint \alpha_{tt}z_\Gamma^2\,d\sigma dt -2\lambda RC \sint \theta\xi z_\Gamma\partial_{\nu}z\,d\sigma dt - \gamma\sint \partial_t z_\Gamma\partial_\nu z\,d\sigma dt \notag
	\\
	&- \lambda RC \sint \theta\xi |\delta\LB z_\Gamma|^2\,d\sigma dt.
\end{align}
Now, thanks to the boundary conditions in \eqref{heat_dbc_z}, we have
\begin{align}\label{MM1}
	- \gamma\sint \partial_t z_\Gamma\partial_\nu z\,d\sigma dt =& -\frac{\gamma}{2}\sint \partial_t z_\Gamma\partial_\nu z\,d\sigma dt -\frac{\gamma}{2}\sint \partial_t z_\Gamma\partial_\nu z\,d\sigma dt \notag
	\\
	=& -\frac{\gamma}{2}\sint \partial_t z_\Gamma\partial_\nu z\,d\sigma dt -\frac{1}{2}\sint |\partial_t z_\Gamma|^2\,d\sigma dt -\frac{\delta}{2}\sint \partial_t z_\Gamma\LB z_\Gamma\,d\sigma dt\notag
	\\
	&+\frac{1}{2}\sint (\beta-R\alpha_t+R\gamma\partial_{\nu}\alpha)z_\Gamma\partial_t z_\Gamma\,d\sigma dt.
\end{align}
In addition, integrating the last term in the right hand side of \eqref{MM1} by parts in time yields 
\begin{align*}
	\frac{1}{2}\sint (\beta-R\alpha_t+R\gamma\partial_{\nu}\alpha)z_\Gamma\partial_t z_\Gamma\,d\sigma dt &= \frac{1}{4}\sint (\beta-R\alpha_t+R\gamma\partial_{\nu}\alpha)\partial_t(z_\Gamma^2)\,d\sigma dt
	\\
	&= -\frac{R}{4}\sint (\gamma\partial_{\nu}\alpha_t-\alpha_{tt})z_\Gamma^2\,d\sigma dt.
\end{align*}
Therefore, from \eqref{JJ-4} we obtain
\begin{align}\label{JJ-5}
	J \geq &\lambda^3R^3C \sint \theta^3\xi^3 \left(1-\frac{1}{\lambda^2 R}-\frac{1}{\lambda^3 R^3}-\frac{1}{\lambda^2 R^2}\right) z_\Gamma^2\,d\sigma dt - \lambda^2RC \sint \theta\xi z_\Gamma\partial_{\nu}z\,d\sigma dt \notag
	\\
	&- \frac{3R}{4} \sint \alpha_{tt}z_\Gamma^2\,d\sigma dt -2\lambda RC \sint \theta\xi z_\Gamma\partial_{\nu}z\,d\sigma dt - \frac{\gamma}{2}\sint \partial_t z_\Gamma\partial_\nu z\,d\sigma dt -\frac{R\gamma}{4} \sint \partial_{\nu}\alpha_t z_\Gamma^2\, \sigma dt \notag
	\\
	&-\frac{1}{2}\sint |\partial_t z_\Gamma|^2\,d\sigma dt -\frac{\delta}{2}\sint \partial_t z_\Gamma\LB z_\Gamma\,d\sigma dt -\lambda RC \sint \theta\xi |\delta\LB z_\Gamma|^2\,d\sigma dt.
\end{align}
Now, by definition of $\alpha$ we have $\partial_{\nu}\alpha_t = -\lambda\theta_t\xi\partial_{\nu}\eta$ and $|\partial_{\nu}\alpha_t|\leq \lambda C\theta^2\xi^3$. Hence, using also \eqref{alpha_t_est} and \eqref{theta_xi_est}, and for $\lambda$ and $R$ large enough, \eqref{JJ-5} becomes
\begin{align}\label{JJ-6}
	J \geq &\lambda^3R^3C \sint \theta^3\xi^3 \left(1-\frac{1}{\lambda^2 R}-\frac{1}{\lambda^3 R^3}-\frac{1}{\lambda^2 R^2}-\frac{1}{\lambda^3 R^2}\right) z_\Gamma^2\,d\sigma dt  \notag
	\\
	&- \lambda^2RC \sint \theta\xi z_\Gamma\partial_{\nu}z\,d\sigma dt -2\lambda RC \sint \theta\xi z_\Gamma\partial_{\nu}z\,d\sigma dt - \lambda RC\sint \partial_t z_\Gamma\partial_\nu z\,d\sigma dt \notag
	\\
	&-\lambda RC \sint \theta\xi \Big(|\partial_t z_\Gamma|^2 + |\delta\LB z_\Gamma|^2 +2\partial_t z_\Gamma\LB z_\Gamma\Big) \,d\sigma dt \notag
	\\
	\geq &\lambda^3R^3C \sint \theta^3\xi^3 \left(1-\frac{1}{\lambda^2 R}-\frac{1}{\lambda^3 R^3}-\frac{1}{\lambda^2 R^2}-\frac{1}{\lambda^3 R^2}\right) z_\Gamma^2\,d\sigma dt - \lambda^2RC \sint \theta\xi z_\Gamma\partial_{\nu}z\,d\sigma dt \notag
	\\
	&-2\lambda RC \sint \theta\xi z_\Gamma\partial_{\nu}z\,d\sigma dt - 3\lambda RC\sint \partial_t z_\Gamma\partial_\nu z\,d\sigma dt -2\lambda RC\gamma\delta \sint \LB z_\Gamma\partial_\nu z\,d\sigma dt \notag 
	\\
	&-\lambda RC \sint \theta\xi \Big(|\partial_t z_\Gamma|^2 + |\delta\LB z_\Gamma|^2 +2\partial_t z_\Gamma\LB z_\Gamma -2\partial_t z_\Gamma\partial_\nu z -2\gamma\delta\LB z_\Gamma\partial_\nu z\Big) \,d\sigma dt. 
\end{align}
Furthermore, using the Young inequality we obtain that
\begin{align*}
	&\;\;- \lambda^2RC \sint \theta\xi z_\Gamma\partial_{\nu}z_\Gamma\,d\sigma dt \geq - \lambda^3RC \sint \theta\xi z_\Gamma^2\,d\sigma dt - \lambda RC \sint \theta\xi |\gamma\partial_{\nu}z|^2\,d\sigma dt
	\\
	&\;\;-2\lambda RC \sint \theta\xi z_\Gamma\partial_{\nu}z\,d\sigma dt \geq -\lambda RC \sint \theta\xi z_\Gamma^2\,d\sigma dt -\lambda RC \sint \theta\xi|\gamma\partial_{\nu}z|^2\,d\sigma dt
	\\
	&\;\;- 3\lambda RC \sint \partial_t z_\Gamma\partial_\nu z\,d\sigma dt \geq - C\sint |\partial_t z_\Gamma|^2\,d\sigma dt - C\sint |\gamma\partial_\nu z|^2\,d\sigma dt
	\\
	&\;\;-2\lambda RC\gamma\delta \sint \LB z_\Gamma\partial_\nu z\,d\sigma dt \geq -\lambda RC \sint |\delta\LB z_\Gamma|^2\,d\sigma dt -\lambda RC \sint |\delta\partial_\nu z|^2\,d\sigma dt.
\end{align*}
By means of these four expressions, from \eqref{JJ-6} we get that
\begin{align*}
	J \geq& \lambda^3R^3C \sint \theta^3\xi^3\left(1-\frac{1}{R^2}-\frac{2}{\lambda^2R^2}-\frac{1}{\lambda^2R}-\frac{1}{\lambda^3R^2}-\frac{1}{\lambda^3R^3}\right) z_\Gamma^2\,d\sigma dt 
	\\
	& -\lambda RC \sint \theta\xi \Big(|\partial_t z_\Gamma|^2 + |\delta\LB z_\Gamma|^2 +|\gamma\partial_\nu z|^2 +2\delta\partial_t z_\Gamma\LB z_\Gamma -2\gamma\partial_t z_\Gamma\partial_\nu z -2\gamma\delta\LB z_\Gamma\partial_\nu z\Big) \,d\sigma dt 
	\\
	=& \lambda^3R^3C \sint \theta^3\xi^3\left(1-\frac{1}{R^2}-\frac{2}{\lambda^2R^2}-\frac{1}{\lambda^2R}-\frac{1}{\lambda^3R^2}-\frac{1}{\lambda^3R^3}\right) z_\Gamma^2\,d\sigma dt 
	\\
	& -\lambda RC \sint \theta\xi |\partial_t z_\Gamma + \delta\LB z_\Gamma -\gamma\partial_\nu z|^2 \,d\sigma dt.
\end{align*}
Hence, for $\lambda$ and $R$ large enough, we finally have
\begin{align*}
	J \geq&\lambda^3R^3C \sint \theta^3\xi^3 z_\Gamma^2\,d\sigma dt -\lambda RC \sint \theta\xi |\partial_t z_\Gamma + \delta\LB z_\Gamma -\gamma\partial_\nu z|^2 \,d\sigma dt.
\end{align*}
Now, collecting all the above estimates, we get from \eqref{ineq_norm} that
 
\begin{align}\label{MM2}
&	\lambda^3R^2C \qint \theta^3\xi^3 z^2\,dxdt + \lambda C\qint \theta\xi|\nabla z|^2\,dxdt + \lambda^2 R^2 C \sint \theta^3\xi^3  z_\Gamma^2\,d\sigma dt \notag
	\\
	&\leq C \sint \theta\xi\left|\partial_tz_\Gamma+\delta\LB z_\Gamma-\gamma\partial_{\nu}z\right|^2\,d\sigma dt.
\end{align}
Recall that $z=\phi e^{-R\alpha}$ so that using the fact that the functions $\alpha$ and $\eta$ are constants at the boundary $\Gamma$, we get that
\begin{equation}\label{MM3}
\begin{cases}
\nabla z=&e^{-R\alpha}\Big(\nabla \phi-R\phi\nabla\alpha\Big)\\
\partial_{\nu}z=&e^{-R\alpha}\Big(\partial_{\nu}\phi-R\phi\partial_\nu\alpha\Big)\\
\partial_tz_\Gamma=& e^{-R\alpha}\Big(\phi_t-R\phi\alpha_t\Big)\\
\LB z_\Gamma=&e^{-R\alpha}\LB\phi_\Gamma.
\end{cases}
\end{equation}
Finally, coming back to the variable $\phi$ in \eqref{MM2} by using \eqref{MM3}, we  obtain the estimate \eqref{carleman} and the proof is finished.
\end{proof}

We conclude this (sub)section with the following remark.

\begin{remark}
{\em We notice that all the above estimates including the Carleman estimate \eqref{carleman} hold for $\delta=0$.
}
\end{remark}

\subsection{The observability inequality}

In this (sub)section, we give the last ingredient needed in the proof of our main result, namely we show an observability inequality.

\begin{proposition}\label{prop-22}
Let $T>0$ be fixed but arbitrary. Then there exists a constant $C_T>0$ such that for every $(\phi_T,\phi_{\Gamma,T})\in \mathbb X^2(\bOm)$, the unique mild solution $\phi$ of the backward system \eqref{heat_dbc_adj} satisfies the estimate
\begin{align}\label{dbc_observ}
\int_{\Omega}|\phi(x,0)|^2\;dx +\int_{\Gamma}|\phi_{\Gamma}(x,0)|^2\;d\sigma \leq C_T \sint \left|\beta(x)\phi_\Gamma(x,t)\right|^2\,d\sigma dt.
\end{align}
\end{proposition}

\begin{proof}
First, assume that $(\phi_T,\phi_{\Gamma,T})\in \mathbb W_\delta^{1,2}(\bOm)$ and let $\phi$ be the unique strong solution of the backward system \eqref{heat_dbc_adj} with final data $(\phi_T,\phi_{\Gamma,T})$.
Let $\lambda\geq\lambda_0$ and $R\geq R_0$ be fixed such that the estimate \eqref{carleman} holds. Then in particular, we have that
\begin{align}\label{IN-CAR}
&	\lambda^3R^2\qint \theta^3\xi^3e^{-2R\alpha}\phi^2\,dxdt + \lambda^2 R^2 \sint \theta^3\xi^3e^{-2R\alpha}  \phi_{\Gamma}^2\,d\sigma dt \notag\\
	&\leq C\sint \theta\xi e^{-2R\alpha}\left|\partial_t\phi_{\Gamma}+\delta\LB\phi_{\Gamma}-\gamma\partial_{\nu}\phi\right|^2\,d\sigma dt.
\end{align}
It is straightforward to check that there exist two positive constants $\mathcal{P}_1$ and $\mathcal{P}_2$ such that
\begin{equation}\label{const}
\begin{cases}
		\theta^3\xi^3 e^{-2R\alpha}\geq\mathcal{P}_1\;\;\;&\mbox{ in } \Omega\times\left[\frac{T}{4},\frac{3T}{4}\right],\\
		\theta\xi e^{-2R\alpha}\leq\mathcal{P}_2 &\mbox{ on }\; \Sigma_T.
\end{cases}
\end{equation}
Using \eqref{const} we get from \eqref{IN-CAR} that there is a constant $C>0$ such that
\begin{align*}
	\int_{\frac{T}{4}}^{\frac{3T}{4}}\int_{\Omega} \phi^2\,dxdt + \int_{\frac{T}{4}}^{\frac{3T}{4}}\int_{\Gamma} \phi_{\Gamma}^2\,d\sigma dt \leq C \sint \left|\partial_t\phi_{\Gamma}+\delta\LB\phi_{\Gamma}-\gamma\partial_{\nu}\phi\right|^2\,d\sigma dt.
\end{align*}
Second, multiplying the first two equations in \eqref{heat_dbc_adj} by $\phi$ and $\phi_\Gamma$,  and integrating over $\Omega$ and $\Gamma$, respectively, we obtain that
\begin{align}\label{IN-1}
\frac{1}{2}\frac{d}{dt}\int_{\Omega} \phi^2\,dx = \gamma\int_{\Omega} |\nabla\phi|^2\,dx  - \gamma\int_{\Gamma}\phi_{\Gamma}\partial_{\nu}\phi\,d\sigma
\end{align}
and
\begin{align}\label{IN-2}
\frac{1}{2}\frac{d}{dt}\int_{\Gamma} \phi_{\Gamma}^2\,d\sigma = \delta\int_{\Gamma}|\nabla_\Gamma \phi_\Gamma|^2\;d\sigma+\gamma\int_{\Gamma}\phi_{\Gamma}\partial_{\nu}\phi\,d\sigma + \int_{\Gamma} \beta\phi_{\Gamma}^2\,d\sigma.
\end{align}
Adding \eqref{IN-1} and \eqref{IN-2} and using \eqref{sobo} we get that there is a constant $C>0$ such that
\begin{align*}
	\frac{1}{2}\frac{d}{dt}\left(\int_{\Omega} \phi^2\,dx + \int_{\Gamma} \phi_{\Gamma}^2\,d\sigma\right) &=\gamma\int_{\Omega} |\nabla\phi|^2\,dx +\delta\int_{\Gamma}|\nabla_\Gamma \phi_\Gamma|^2\;d\sigma + \int_{\Gamma} \beta\phi_{\Gamma}^2\,d\sigma \\
	&\geq C\left(\int_{\Omega}\phi^2\,dx + \int_{\Gamma} \phi_{\Gamma}^2\,d\sigma\right).
\end{align*}
This clearly implies that
\begin{align}\label{IN-E}
	e^{CT}\left(\int_{\Omega} |\phi(x,0)|^2\,dx + \int_{\Gamma} |\phi_{\Gamma}(x,0)|^2\,d\sigma\right)\leq \int_{\Omega} \phi^2\,dx + \int_{\Gamma} \phi_{\Gamma}^2\,d\sigma.
\end{align}
Integrating \eqref{IN-E} in time from $\frac T4$ to $\frac{3T}{4}$ we get that
\begin{align*}
	\frac{T}{2}e^{CT}\left(\int_{\Omega}| \phi(x,0)|^2\,dx + \int_{\Gamma} |\phi_{\Gamma}(x,0)|^2\,d\sigma\right) \leq \int_{\frac{T}{4}}^{\frac{3T}{4}}\left(\int_{\Omega} \phi^2\,dx + \int_{\Gamma} \phi_{\Gamma}^2\,d\sigma\right)\,dt.
\end{align*}
Thus, we  obtain the observability inequality
\begin{align*}
	\int_{\Omega} |\phi(x,0)|^2\,dx + \int_{\Gamma} |\phi_{\Gamma}(x,0)|^2\,d\sigma \leq& \frac{2C_1}{T}e^{-CT} \sint \left|\partial_t\phi_{\Gamma}+\delta\LB\phi_{\Gamma}-\gamma\partial_{\nu}\phi\right|^2\,d\sigma dt\\
	=&\frac{2C_1}{T}e^{-CT} \sint |\beta\phi_\Gamma|^2\,d\sigma dt
\end{align*}	
for a strong solution.

Finally, let $(\phi_T,\phi_{\Gamma,T})\in \mathbb X^2(\bOm)$ and $\phi$ the unique mild solution of the backward system \eqref{heat_dbc_adj}. Let $(\phi_{T,n},\phi_{\Gamma,T,n})\in \mathbb W_\delta^{1,2}(\bOm)$ be a sequence which converges to $(\phi_T,\phi_{\Gamma,T})$ in $\mathbb X^2(\bOm)$. Then the strong solution $(\phi_n,\phi_{n,\Gamma})$ with final data $(\phi_{T,n},\phi_{\Gamma,T,n})$ converges in $C([0,T];\mathbb X^2(\bOm))$ to the mild solution $(\phi,\phi_\Gamma)$ with final data $(\phi_T,\phi_{\Gamma,T})$. It follows from the first part of the proof that
\begin{align}\label{seq}
	\int_{\Omega} \phi_n(x,0)^2\,dx + \int_{\Gamma} \phi_{n,\Gamma}^2(x,0)\,d\sigma \le \frac{2C_1}{T}e^{-CT} \sint |\beta\phi_{n,\Gamma}|^2\,d\sigma dt.
\end{align}
Taking the limit of \eqref{seq} as $n\to\infty$ and using the above mentioned convergence, we get the estimate \eqref{dbc_observ}
and the proof is finished.
\end{proof}

Now we are ready to give the proof of the main result of the paper.

\begin{proof}[\bf Proof of Theorem \ref{dbc_control_thm}]
We use some ideas of the proof of \cite[Theorem 4.2]{MS2013}. Let us introduce the following weighted $L^2$-spaces
\begin{align*}%\label{weighted_L2}
	\ZQ:=\left\{f\in L^2(\Omega_T):\, e^{R\alpha}(\theta\xi)^{-3/2}f\in L^2(\Omega_T)\right\}, \;\langle f_1,f_2\rangle_{\ZQ} = \qint f_1f_2e^{2R\alpha}(\theta\xi)^{-3}\,dxdt, \nonumber
	\\
	\ZS:=\left\{g\in L^2(\Sigma_T):\, e^{R\alpha}(\theta\xi)^{-3/2}g\in L^2(\Sigma_T)\right\}, \;\langle g_1,g_2\rangle_{\ZS} = \sint g_1g_2e^{2R\alpha}(\theta\xi)^{-3}\,d\sigma  dt.
\end{align*}
The boundary controllability of the system \eqref{heat_dbc} will follow by a duality argument. At this purpose, let us define the bounded linear operator $\TT:L^2(\Sigma_T)\to\mathbb X^2(\bOm)$ by
\begin{align*}
	\TT v:=\int_0^T e^{(T-\tau)A_\delta}(0,-v(\tau))\,d\tau,
\end{align*}
where we recall that $(e^{tA_\delta})_{t\ge 0}$ is the strongly continuous submarkovian semigroup generated by the operator $A_\delta$ in $\mathbb X^2(\bOm)$.
 Using the continuous embedding $\ZQ\times\ZS\hookrightarrow L^2(\Omega_T)\times L^2(\Sigma_T)$, we also introduce the bounded linear operator $\sop:\mathbb X^2(\bOm)\times\ZQ\times\ZS\to\mathbb X^2(\bOm)$ given by
\begin{align*}
	\sop(U_0,f,g):= e^{TA_\delta}U_0 + \int_0^T e^{(T-\tau)A_\delta}(f(\cdot,\tau),g(\cdot,\tau))\,d\tau.
\end{align*}
We claim that
\begin{align}\label{claim}
 \sop(U_0,0,g)-\TT v=(u(\cdot,T),u_\Gamma(\cdot,T)),
 \end{align}
 where  $u$ is the unique mild solution of the system
\begin{align}\label{heat_dbc2}
\begin{cases}
u_t - \gamma\Delta u  = 0 \;\;\;&\mbox{ in }\;\;\Omega\times(0,T)\\
\partial_tu_{\Gamma} -\delta\Delta_\Gamma u_\Gamma+ \gamma\partial_{\nu}u + \beta u = g+v \;\;&\mbox{ on }\;\;\Gamma\times(0,T)\\
\left.(u,u_{\Gamma})\right|_{t=0} = (u_0,u_{\Gamma,0}) &\mbox{ in }\;\;\Omega\times\Gamma.
\end{cases}
\end{align}
In fact, using Proposition \ref{ex-sol} and the representation of mild solution given in Definition \ref{def-sol}, we have that
\begin{align*}
	\sop(U_0,0,g)-\TT v &= e^{TA_\delta}U_0 + \int_0^T e^{(T-\tau)A_\delta}(0,g(\cdot,\tau))\,d\tau - \int_0^T e^{(T-\tau)A_\delta}(0,-v(\cdot,\tau))\,d\tau \\
	&= e^{TA_\delta}U_0 + \int_0^T e^{(T-\tau)A_\delta}(0,g(\cdot,\tau)+v(\cdot,\tau))\,d\tau = (u(\cdot,T),u_\Gamma(\cdot,T)),
\end{align*}
and we have shown the claim.
Furthermore it is straightforward to check that the adjoint operator $\TT^*:\mathbb X^2(\bOm)\to L^2(\Sigma_T)$ is given for $\Phi_T:=(\phi_T,\phi_{\Gamma,T})$ by
\begin{align*}
	\TT^*\Phi_T=- \phi_{\Gamma}
\end{align*}
where $(\phi(t),\phi_{\Gamma}(t))=e^{(T-t)A_\delta}(\phi_T,\phi_{\Gamma,T})$ is the solution of the backward system \eqref{heat_dbc_adj} with final data $(\phi_T,\phi_{\Gamma,T})$, while the adjoint operator $\sop^*:\mathbb X^2(\bOm)\to\mathbb X^2(\bOm)\times\ZQ\times\ZS$ of $\sop$ is given by 
\begin{align*}
	\sop^*\Phi_T=\left((\phi(0),\phi_{\Gamma}(0)),\,e^{-2R\alpha}(\theta\xi)^3\phi,\,e^{-2R\alpha}(\theta\xi)^3\phi_{\Gamma}\right).
\end{align*}
Now, the observability inequality  \eqref{dbc_observ} and the Carleman estimate \eqref{carleman} imply that
\begin{align}\label{IN-F}
	\norm{\sop^*\phi_T}{\mathbb X^2(\bOm)\times\ZQ\times\ZS}^2 = &\norm{(\phi(\cdot,0),\phi_\Gamma(\cdot,0))}{\mathbb X^2(\bOm)}^2 + \qint \theta^3\xi^3e^{-2R\alpha}\phi^2\,dxdt  \notag\\
	&+\sint \theta^3\xi^3e^{-2R\alpha}\phi_{\Gamma}^2\,d\sigma  dt \notag
	\\
	\le& C_T \sint \beta^2|\phi_\Gamma|^2\,d\sigma dt \le  C_T\|\beta\|_{L^\infty(\Gamma)}^2\norm{\TT^*\Phi_T}{L^2(\Sigma_T)},
\end{align}
at first for $(\phi_T,\phi_{\Gamma,T})\in\mathbb W_\delta^{1,2}(\bOm)$ and then for $(\phi_T,\phi_{\Gamma,T})\in\mathbb X^2(\bOm)$ by using an approximation argument as at the end of the proof of Proposition \ref{prop-22}.
By \cite[Theorem IV.2.2]{Zab}, the estimate \eqref{IN-F} implies that $\textrm{Im}(\sop)\subset\textrm{Im}(\TT)$. This shows that for every $(U_0,f,g)\in\mathbb X^2(\bOm)\times\ZQ\times\ZS$, there exists a control $v\in L^2(\Sigma_T)$ such that $\sop(U_0,f,g)=\TT v$. Therefore, 
\begin{align*}
(u(\cdot,T),u_{\Gamma}(\cdot,T))=\sop(U_0,f,g)-\TT v=(0,0)
\end{align*}
 and the proof is finished.
\end{proof}

\end{document}